\documentclass{amsart}
\usepackage{hyperref}

\newtheorem{theorem}{Theorem}[section]
\newtheorem{lemma}[theorem]{Lemma}

\newcommand{\cB}{\mathcal{B}}
\newcommand{\cR}{\mathcal{R}}

\newcommand{\C}{\mathbb{C}}
\newcommand{\R}{\mathbb{R}}
\newcommand{\Z}{\mathbb{Z}}

\newcommand{\alg}{\mathrm{alg}\,}
\newcommand{\eps}{\varepsilon}
\newcommand{\im}{\mathrm{Im}\,}
\newcommand{\Ker}{\mathrm{Ker}\,}

\title[Semi-Fredholm singular integral operators]{%
Semi-Fredholm singular integral operators
with piecewise continuous coefficients
on weighted variable Lebesgue spaces
are Fredholm}
\author{Alexei Yu. Karlovich}
\address{%
Departamento de Matem\'atica,
Instituto Superior T\'ecnico,
Av. Rovisco Pais 1,
1049-001, Lisbon, Portugal}
\email{akarlov@math.ist.utl.pt}
\thanks{The author is partially supported by F.C.T. (Portugal) grants
SFRH/BPD/11619/2002 and FCT/FEDER/POCTI/MAT/59972/2004}


\begin{document}
\begin{abstract}
Suppose $\Gamma$ is a Carleson Jordan curve with logarithmic whirl points,
$\varrho$ is a Khvedelidze weight, $p:\Gamma\to(1,\infty)$ is a continuous
function satisfying $|p(\tau)-p(t)|\le -\mathrm{const}/\log|\tau-t|$ for
$|\tau-t|\le 1/2$, and $L^{p(\cdot)}(\Gamma,\varrho)$ is a weighted generalized
Lebesgue space with variable exponent. We prove that all semi-Fredholm
operators in the algebra of singular integral operators with $N\times N$
matrix piecewise continuous coefficients are Fredholm on
$L_N^{p(\cdot)}(\Gamma,\varrho)$.
\end{abstract}
\maketitle
\section{Introduction}
Let $X$ be a Banach space and $\cB(X)$ be the Banach algebra of all bounded
linear operators on $X$. An operator $A\in\cB(X)$ is said to be $n$-normal
(resp. $d$-normal) if its image $\im A$ is closed in $X$ and the defect number
$n(A;X):=\dim\Ker A$ (resp. $d(A;X):=\dim\Ker A^*$) is finite. An operator $A$
is said to be semi-Fredholm on $X$ if it is $n$-normal or $d$-normal. Finally,
$A$ is said to be Fredholm if it is simultaneously $n$-normal and $d$-normal.
Let $N$ be a positive integer. We denote by $X_N$ the direct sum of $N$ copies of
$X$ with the norm
\[
\|f\|=\|(f_1,\dots,f_N)\|:=(\|f_1\|^2+\dots+\|f_N\|^2)^{1/2}.
\]

Let $\Gamma$ be a Jordan curve, that is, a curve that is homeomorphic to a circle.
We suppose that $\Gamma$ is rectifiable. We equip $\Gamma$ with Lebesgue length
measure $|d\tau|$ and the counter-clockwise orientation. The \textit{Cauchy
singular integral} of $f\in L^1(\Gamma)$ is defined by
\[
(Sf)(t):=\lim_{R\to 0}\frac{1}{\pi i}\int_{\Gamma\setminus\Gamma(t,R)}
\frac{f(\tau)}{\tau-t}d\tau
\quad (t\in\Gamma),
\]
where $\Gamma(t,R):=\{\tau\in\Gamma:|\tau-t|<R\}$ for $R>0$.
David \cite{David84} (see also \cite[Theorem~4.17]{BK97}) proved that the
Cauchy singular integral generates the bounded operator $S$ on the Lebesgue
space $L^p(\Gamma)$, $1<p<\infty$, if and only if $\Gamma$ is a Carleson
(Ahlfors-David regular) curve, that is,
\[
\sup_{t\in\Gamma}\sup_{R>0}\frac{|\Gamma(t,R)|}{R}<\infty,
\]
where $|\Omega|$ denotes the measure of a measurable set $\Omega\subset\Gamma$.
We can write $\tau-t=|\tau-t|e^{i\arg(\tau-t)}$
for $\tau\in\Gamma\setminus\{t\}$,
and the argument can be chosen so that it is continuous on $\Gamma\setminus\{t\}$.
It is known \cite[Theorem~1.10]{BK97} that for an arbitrary Carleson curve the
estimate
\[
\arg(\tau-t)=O(-\log|\tau-t|)\quad (\tau\to t)
\]
holds for every $t\in\Gamma$. One says
that a Carleson curve $\Gamma$ satisfies the \textit{logarithmic whirl condition}
at $t\in\Gamma$ if
\begin{equation}\label{eq:spiralic}
\arg(\tau-t)=-\delta(t)\log|\tau-t|+O(1)\quad (\tau\to t)
\end{equation}
with some $\delta(t)\in\R$. Notice that all piecewise smooth curves satisfy this
condition at each point and, moreover, $\delta(t)\equiv 0$. For more information
along these lines, see \cite{BK95}, \cite[Chap.~1]{BK97}, \cite{BK01}.

Let $t_1,\dots,t_m\in\Gamma$ be pairwise distinct points. Consider the
Khvedelidze weight
\[
\varrho(t):=\prod_{k=1}^m|t-t_k|^{\lambda_k}
\quad(\lambda_1,\dots,\lambda_m\in\R).
\]
Suppose $p:\Gamma\to(1,\infty)$ is a continuous function. Denote by
$L^{p(\cdot)}(\Gamma,\varrho)$ the set of all measurable complex-valued functions
$f$ on $\Gamma$ such that
\[
\int_\Gamma |f(\tau)\varrho(\tau)/\lambda|^{p(\tau)}|d\tau|<\infty
\]
for some $\lambda=\lambda(f)>0$. This set becomes a Banach space when equipped
with the Luxemburg-Nakano norm
\[
\|f\|_{p(\cdot),\varrho}:=\inf\left\{\lambda>0:
\int_\Gamma |f(\tau)\varrho(\tau)/\lambda|^{p(\tau)}|d\tau|\le 1\right\}.
\]
If $p$ is constant, then $L^{p(\cdot)}(\Gamma,\varrho)$ is nothing else than
the weighted Lebesgue space. Therefore, it is natural to refer to
$L^{p(\cdot)}(\Gamma,\varrho)$ as a \textit{weighted generalized Lebesgue space
with variable exponent} or simply as weighted variable Lebesgue spaces.
This is a special case of Musielak-Orlicz spaces \cite{Musielak83}.
Nakano \cite{Nakano50} considered these spaces (without weights) as
examples of so-called modular spaces, and sometimes the spaces
$L^{p(\cdot)}(\Gamma,\varrho)$ are referred to as weighted Nakano spaces.

If $S$ is bounded on $L^{p(\cdot)}(\Gamma,\varrho)$, then from
\cite[Theorem~6.1]{Karlovich03} it follows that $\Gamma$ is a Carleson curve.
The following result is announced in \cite[Theorem~7.1]{Kokilashvili05}
and in \cite[Theorem~D]{KS05}. Its full proof is published in \cite{KPS-Simonenko}.
\begin{theorem}\label{th:KPS}
Let $\Gamma$ be a Carleson Jordan curve and $p:\Gamma\to(1,\infty)$
be a continuous function satisfying
\begin{equation}\label{eq:Dini-Lipschitz}
|p(\tau)-p(t)|\le -A_\Gamma/\log|\tau-t|
\quad\mbox{whenever}\quad
|\tau-t|\le1/2,
\end{equation}
where $A_\Gamma$ is a positive constant depending only on $\Gamma$.
The Cauchy singular integral operator $S$ is bounded on $L^{p(\cdot)}(\Gamma,\varrho)$
if and only if
\begin{equation}\label{eq:Khvedelidze}
0<1/p(t_k)+\lambda_k<1
\quad\mbox{for all}\quad
k\in\{1,\dots,m\}.
\end{equation}
\end{theorem}
We define by $PC(\Gamma)$ as the set of all $a\in L^\infty(\Gamma)$ for which
the one-sided limits
\[
a(t\pm 0):=\lim_{\tau\to t\pm 0}a(\tau)
\]
exist at each point $t\in\Gamma$; here $\tau\to t-0$ means that $\tau$
approaches $t$ following the orientation of $\Gamma$, while $\tau\to t+0$
means that $\tau$ goes to $t$ in the opposite direction. Functions in $PC(\Gamma)$
are called piecewise continuous functions.

The operator $S$ is defined on $L_N^{p(\cdot)}(\Gamma,\varrho)$ elementwise.
We let stand $PC_{N\times N}(\Gamma)$ for the algebra of all $N\times N$
matrix functions with entries in $PC(\Gamma)$. Writing the elements of
$L_N^{p(\cdot)}(\Gamma,\varrho)$ as columns, we can define the multiplication
operator $aI$ for $a\in PC_{N\times N}(\Gamma)$ as multiplication by the matrix
function $a$. Let $\alg(S,PC;L_N^{p(\cdot)}(\Gamma,\varrho))$ denote the smallest
closed subalgebra of $\cB(L_N^{p(\cdot)}(\Gamma,\varrho))$ containing the
operator $S$ and the set $\{aI:a\in PC_{N\times N}(\Gamma)\}$.

For the case of piecewise Lyapunov curves $\Gamma$ and constant exponent $p$,
a Fredholm criterion for an arbitrary operator $A\in \alg(S,PC;L_N^p(\Gamma,\varrho))$
was obtained by Gohberg and Krupnik \cite{GK71} (see also \cite{GK92} and
\cite{Krupnik87}). Spitkovsky \cite{Spitkovsky92} established a Fredholm criterion
for the operator $aP+Q$, where $a\in PC_{N\times N}(\Gamma)$ and
\[
P:=(I+S)/2,
\quad
Q:=(I-S)/2,
\]
on the space $L_N^p(\Gamma,w)$, where $\Gamma$ is a smooth curve and $w$ is
an arbitrary Muckenhoupt weight. He also proved that if $aP+Q$ is semi-Fredholm
on $L_N^p(\Gamma,w)$, then it is automatically Fredholm on $L_N^p(\Gamma,w)$.
These results were extended to the case of an arbitrary operator
$A\in \alg(S,PC;L_N^p(\Gamma,w))$ in \cite{GKS93}. The Fredholm theory for singular
integral operators with piecewise continuous coefficients on Lebesgue spaces
with arbitrary Muckenhoupt weights on arbitrary Carleson curves curves was
accomplished in a series of papers by B\"ottcher and Yu.~Karlovich. It is
presented in their monograph \cite{BK97} (see also the nice survey \cite{BK01}).

The study of singular integral operators with discontinuous coefficients
on generalized Lebesgue spaces with variable exponent was started in \cite{KS03,KPS05}.
The results of \cite{BK97} are partially extended to the case of weighted
generalized Lebesgue spaces with variable exponent in
\cite{Karlovich03,Karlovich05,Karlovich06}. Suppose $\Gamma$ is a Carleson curve
satisfying the logarithmic whirl condition (\ref{eq:spiralic}) at each point
$t\in\Gamma$, $\varrho$ is a Khvedelidze weight, and $p$ is a variable exponent
as in Theorem~\ref{th:KPS}. Under these assumptions, a Fredholm criterion for an
arbitrary operator $A$ in the algebra $\alg(S,PC;L_N^{p(\cdot)}(\Gamma,\varrho))$
is obtained in \cite[Theorem~5.1]{Karlovich05} by using the Allan-Douglas
local principle \cite[Section~1.35]{BS06} and the two projections theorem
\cite{FRS93}. However, this approach does not allow us to get additional
information about semi-Fredholm and Fredholm operators in this algebra. For
instance, to obtain an index formula for Fredholm operators in this algebra,
we need other means (see, e.g., \cite[Section~6]{Karlovich06}).
Following the ideas of \cite{GK71,Spitkovsky92,GKS93}, in this paper
we present a self-contained proof of the following result.
\begin{theorem}\label{th:main}
Let $\Gamma$ be a Carleson Jordan curve satisfying the logarithmic whirl
condition {\rm(\ref{eq:spiralic})} at each point $t\in\Gamma$, let
$p:\Gamma\to(1,\infty)$ be a continuous function satisfying
{\rm(\ref{eq:Dini-Lipschitz})}, and let $\varrho$ be a Khvedelidze weight
satisfying {\rm(\ref{eq:Khvedelidze})}. If an operator in the
algebra $\alg(S,PC;L_N^{p(\cdot)}(\Gamma,\varrho))$ is semi-Fredholm, then
it is Fredholm.
\end{theorem}
The paper is organized as follows. Section~\ref{sec:semi-Fredholm}
contains general results on semi-Fredholm operators. Some auxiliary
results on singular integral operators acting on $L^{p(\cdot)}(\Gamma,\varrho)$
are collected in Section~\ref{sec:SIO}. In Section~\ref{sec:closed-range},
we prove a criterion guaranteeing that $aP+Q$, where $a\in PC(\Gamma)$,
has closed image in $L^{p(\cdot)}(\Gamma,\varrho)$. This criterion is
intimately related with a Fredholm criterion for $aP+Q$ proved in
\cite{Karlovich05}. Notice that we are able to prove both results for
Carleson Jordan curves which satisfy the additional condition (\ref{eq:spiralic}).
Section~\ref{sec:necessity} contains the proof of the fact that if
the operator $aP+bQ$
is semi-Fredholm on $L_N^{p(\cdot)}(\Gamma,\varrho)$, then the coefficients
$a$ and $b$ are invertible in the algebra $L_{N\times N}^\infty(\Gamma)$.
In Section~\ref{sec:SIO-matrix}, we prove that the semi-Fredholmness and
Fredholmness of $aP+bQ$ on $L_N^{p(\cdot)}(\Gamma,\varrho)$, where $a$ and
$b$ are piecewise continuous matrix functions, are equivalent. In
Section~\ref{sec:algebra}, we extend this result to the sums of products
of operators of the form $aP+bQ$ by using the procedure of linear dilation.
Since these sums are dense in $\alg(S,PC;L_N^{p(\cdot)}(\Gamma,\varrho))$,
Theorem~\ref{th:main} follows from stability properties of semi-Fredholm
operators.
\section{General results on semi-Fredholm and Fredholm operators}
\label{sec:semi-Fredholm}
\subsection{The Atkinson and Yood theorems}
For a Banach space $X$, let $\Phi(X)$ be the set of all Fredholm operators
on $X$ and let $\Phi_+(X)$ (resp. $\Phi_-(X)$) denote the set of all
$n$-normal (resp. $d$-normal) operators $A\in\cB(X)$ such that $d(A;X)=+\infty$
(resp. $n(A;X)=+\infty$).
\begin{theorem}\label{th:Atkinson-Yood}
Let $X$ be a Banach space and $K$ be a compact operator on $X$.
\begin{enumerate}
\item[{\rm(a)}] If $A,B\in\Phi(X)$, then $AB\in\Phi(X)$ and $A+K\in\Phi(X)$.
\item[{\rm(b)}] If $A,B\in\Phi_\pm(X)$, then $AB\in\Phi_\pm(X)$ and $A+K\in\Phi_\pm(X)$.
\item[{\rm(c)}] If $A\in\Phi(X)$ and $B\in\Phi_\pm(X)$, then $AB\in\Phi_\pm(X)$
and $BA\in\Phi_\pm(X)$.
\end{enumerate}
\end{theorem}
Part (a) is due to Atkinson, parts (b) and (c) were obtained by Yood.
For a proof, see e.g. \cite[Chap.~4, Sections 6 and 15]{GK92}.
\begin{theorem}[see e.g. \cite{GK92}, Chap.~4, Theorem~7.1]
\label{th:regularization}
Let $X$ be a Banach space. An operator $A\in\cB(X)$ is Fredholm if and only
if there exists an operator $R\in\cB(X)$ such that $AR-I$ and $RA-I$ are
compact.
\end{theorem}
\subsection{Stability of semi-Fredholm operators}
\begin{theorem}[see e.g. \cite{GK92}, Chap. 4, Theorems~6.4, 15.4]
\label{th:stability}
Let $X$ be a Banach space.
\begin{enumerate}
\item[(a)]
If $A\in\Phi(X)$, then there exists an $\eps=\eps(A)>0$ such that
$A+D\in\Phi(X)$ whenever $\|D\|_{\cB(X)}<\eps$.

\item[(b)]
If $A\in\Phi_\pm(X)$, then there exists an $\eps=\eps(A)>0$ such that
$A+D\in\Phi_\pm(X)$ whenever $\|D\|_{\cB(X)}<\eps$.
\end{enumerate}
\end{theorem}
\begin{lemma}\label{le:stability}
Let $X$ be a Banach space. Suppose $A$ is a semi-Fredholm operator on $X$
and $\|A_n-A\|_{\cB(X)}\to 0$ as $n\to\infty$. If the operators $A_n$ are
Fredholm on $X$ for all sufficiently large $n$, then $A$ is Fredholm, too.
\end{lemma}
\begin{proof}
Assume $A$ is semi-Fredholm, but not Fredholm. Then either
$A\in\Phi_-(X)$ or $A\in\Phi_+(X)$. By Theorem~\ref{th:stability}(b), either
$A_n\in\Phi_-(X)$ or $A_n\in\Phi_+(X)$ for all sufficiently large $n$. That is,
$A_n$ are not Fredholm. This contradicts the hypothesis.
\end{proof}
We refer to the monograph by Gohberg and Krupnik \cite{GK92} for a detailed
presentation of the theory of semi-Fredholm operators on Banach spaces.
\subsection{Semi-Fredholmness of block operators}
Let a Banach space $X$ be represented as the direct sum of its subspaces
$X=X_1\dot{+}X_2$. Then every operator $A\in\cB(X)$ can be written in the
form of an operator matrix
\[
A=\left[\begin{array}{cc}
A_{11} & A_{12} \\ A_{21} & A_{22}
\end{array}\right],
\]
where $A_{ij}\in\cB(X_j,X_i)$ and $i,j=1,2$. The following result is stated
without proof in \cite{Spitkovsky80}. Its proof is given in \cite{Spitkovsky81}
(see also \cite[Theorem~1.12]{LS87}).
\begin{theorem}\label{th:Spitkovsky}
\begin{enumerate}
\item[{\rm(a)}]
Suppose $A_{21}$ is compact. If $A$ is $n$-normal ($d$-normal), then $A_{11}$
(resp. $A_{22}$) is $n$-normal (resp. $d$-normal).
\item[{\rm(b)}]
Suppose $A_{12}$ or $A_{21}$ is compact. If $A_{11}$ (resp. $A_{22}$) is Fredholm,
then $A_{22}$ (resp. $A_{11}$) is $n$-normal, $d$-normal, Fredholm if and only
if $A$ has the corresponding property.
\end{enumerate}
\end{theorem}
\section{Singular integrals on weighted variable Lebesgue spaces}
\label{sec:SIO}
\subsection{Duality of weighted variable Lebesgue spaces}
Suppose $\Gamma$ is a rectifiable Jordan curve and $p:\Gamma\to(1,\infty)$
is a continuous function. Since $\Gamma$ is compact, we have
\[
1<\underline{p}:=\min_{t\in\Gamma}p(t),
\quad
\overline{p}:=\max_{t\in\Gamma}p(t)<\infty.
\]
Define the conjugate exponent $p^*$ for the exponent $p$ by
\[
p^*(t):=\frac{p(t)}{p(t)-1}\quad (t\in\Gamma).
\]
Suppose $\varrho$ is a Khvedelidze weight. If $\varrho\equiv 1$, then we will
write $L^{p(\cdot)}(\Gamma)$ and $\|\cdot\|_{p(\cdot)}$ instead of
$L^{p(\cdot)}(\Gamma,1)$ and $\|\cdot\|_{p(\cdot),1}$, respectively.
\begin{theorem}[see \cite{KR91}, Theorem~2.1]
\label{th:Hoelder}
If $f\in L^{p(\cdot)}(\Gamma)$ and $g\in L^{p^*(\cdot)}(\Gamma)$, then
$fg\in L^1(\Gamma)$ and
\[
\|fg\|_1\le (1+1/\underline{p}-1/\overline{p})\,\|f\|_{p(\cdot)}\|g\|_{p^*(\cdot)}.
\]
\end{theorem}
The above H\"older type inequality in the more general setting of Musielak-Orlicz
spaces is contained in \cite[Theorem~3.13]{Musielak83}.
\begin{theorem}\label{th:linear-functional}
The general form of a linear functional on $L^{p(\cdot)}(\Gamma,\varrho)$
is given by
\[
G(f)=\int_\Gamma f(\tau)\overline{g(\tau)}\,|d\tau|
\quad
(f\in L^{p(\cdot)}(\Gamma,\varrho)),
\]
where  $g\in L^{p^*(\cdot)}(\Gamma,\varrho^{-1})$. The norms in the dual space
$[L^{p(\cdot)}(\Gamma,\varrho)]^*$ and in the space $L^{p^*(\cdot)}(\Gamma,\varrho^{-1})$
are equivalent.
\end{theorem}
The above result can be extracted from \cite[Corollary~13.14]{Musielak83}.
For the case $\varrho=1$, see also \cite[Corollary~2.7]{KR91}.
\subsection{Smirnov classes and Hardy type subspaces}
Let $\Gamma$ be a rectifiable Jordan curve in the complex plane $\C$. We denote
by $D_+$ and $D_-$ the bounded and unbounded components of $\C\setminus\Gamma$,
respectively. We orient $\Gamma$ counter-clockwise. Without loss of generality
we assume that $0\in D_+$. A function $f$ analytic in $D_+$ is said to be in
the Smirnov class $E^q(D_+)$ ($0<q<\infty$) if there exists a sequence of
rectifiable Jordan curves $\Gamma_n$ in $D_+$ tending to the boundary $\Gamma$
in the sense that $\Gamma_n$ eventually surrounds each compact subset of $D_+$
such that
\begin{equation}\label{eq:Smirnov}
\sup_{n\ge 1}\int_{\Gamma_n}|f(z)|^q|dz|<\infty.
\end{equation}
The Smirnov class $E^q(D_-)$ is the set of all analytic functions in
$D_-\cup\{\infty\}$ for which \eqref{eq:Smirnov} holds with some sequence
of curves $\Gamma_n$ tending to the boundary in the sense that every compact
subset of $D_-\cup\{\infty\}$ eventually lies outside $\Gamma_n$. We denote
by $E_0^q(D_-)$ the set of functions in $E^q(D_-)$ which vanish at infinity.
The functions in $E^q(D_\pm)$ have nontangential boundary values almost
everywhere on $\Gamma$ (see, e.g. \cite[Theorem~10.3]{Duren70}). We will
identify functions in $E^q(D_\pm)$ with their nontangential boundary values.
The next result is a consequence of the H\"older inequality.
\begin{lemma}\label{le:Hoelder-Smirnov}
Let $\Gamma$ be a rectifiable Jordan curve. Suppose $0<q_1,q_2,\dots, q_r<\infty$
and $f_j \in E^{q_j}(D_\pm)$ for all $j\in\{1,2,\dots,r\}$. Then
$f_1f_2\dots f_r\in E^q(D_\pm)$, where
$1/q=1/q_1+1/q_2+\dots+1/q_r$.
\end{lemma}
Let $\cR$ denote the set of all rational functions without poles on $\Gamma$.
\begin{theorem}\label{th:Lusin-Privalov}
Let $\Gamma$ be a rectifiable Jordan curve and $0<q<\infty$. If $f$ belongs
to $E^q(D_\pm)+\cR$ and its nontangential boundary values vanish on a subset
$\gamma\subset\Gamma$ of positive measure, then $f$ vanishes identically in
$D_\pm$.
\end{theorem}
This result follows from the Lusin-Privalov
theorem for meromorphic functions (see, e.g. \cite[p.~292]{Privalov50}).

We refer to the monographs by Duren \cite{Duren70} and Privalov \cite{Privalov50}
for a detailed exposition of the theory of Smirnov classes over domains
with rectifiable boundary.
\begin{lemma}\label{le:basic}
Let $\Gamma$ be a Carleson Jordan curve, let $p:\Gamma\to(1,\infty)$ be a
continuous function satisfying {\rm(\ref{eq:Dini-Lipschitz})}, and let $\varrho$
be a Khvedelidze weight satisfying {\rm(\ref{eq:Khvedelidze})}.
Then $P^2=P$ and $Q^2=Q$ on $L^{p(\cdot)}(\Gamma,\varrho)$.
\end{lemma}
This result follows from Theorem~\ref{th:KPS} and \cite[Lemma~6.4]{Karlovich03}.

In view of Lemma~\ref{le:basic}, the Hardy type subspaces
$PL^{p(\cdot)}(\Gamma,\varrho)$, $QL^{p(\cdot)}(\Gamma,\varrho)$, and
$QL^{p(\cdot)}(\Gamma,\varrho)\stackrel{\cdot}{+}\C$ of $L^{p(\cdot)}(\Gamma,\varrho)$
are well defined. Combining Theorem~\ref{th:KPS} and \cite[Lemma~6.9]{Karlovich03}
we obtain the following.
\begin{lemma}\label{le:Hardy-Smirnov}
Let $\Gamma$ be a Carleson Jordan curve, let $p:\Gamma\to(1,\infty)$ be a
continuous function satisfying {\rm(\ref{eq:Dini-Lipschitz})}, and let $\varrho$
be a Khvedelidze weight satisfying {\rm(\ref{eq:Khvedelidze})}. Then
\[
\begin{split}
E^1(D_+)\cap L^{p(\cdot)}(\Gamma,\varrho)
&=
PL^{p(\cdot)}(\Gamma,\varrho),
\\
E_0^1(D_-)\cap L^{p(\cdot)}(\Gamma,\varrho)
&=
QL^{p(\cdot)}(\Gamma,\varrho),
\\
E^1(D_-)\cap L^{p(\cdot)}(\Gamma,\varrho)
&=
QL^{p(\cdot)}(\Gamma,\varrho)\stackrel{\cdot}{+}\C.
\end{split}
\]
\end{lemma}
\subsection{Singular integral operators on the dual space}
For a rectifiable Jordan curve $\Gamma$ we have $d\tau=e^{i\Theta_\Gamma(\tau)}|d\tau|$
where $\Theta_\Gamma(\tau)$ is the angle between the positively oriented real
axis and the naturally oriented tangent of $\Gamma$ at $\tau$ (which exists
almost everywhere). Let the operator $H_\Gamma$ be defined by
$(H_\Gamma\varphi)(t)=e^{-i\Theta_\Gamma(t)}\overline{\varphi(t)}$ for $t\in\Gamma$.
Note that $H_\Gamma$ is additive but
$H_\Gamma(\alpha\varphi)=\overline{\alpha}H_\Gamma\varphi$ for $\alpha\in\C$.
Evidently, $H_\Gamma^2=I$.

From Theorem~\ref{th:KPS} and \cite[Lemma~6.6]{Karlovich03} we get the following.
\begin{lemma}\label{le:S-adjoint}
Let $\Gamma$ be a Carleson Jordan curve, let $p:\Gamma\to(1,\infty)$ be a
continuous function satisfying {\rm(\ref{eq:Dini-Lipschitz})}, and let $\varrho$
be a Khvedelidze weight satisfying {\rm(\ref{eq:Khvedelidze})}. The adjoint
operator of $S\in\cB(L^{p(\cdot)}(\Gamma,\varrho))$ is the operator
$-H_\Gamma SH_\Gamma\in\cB(L^{p^*(\cdot)}(\Gamma,\varrho^{-1}))$.
\end{lemma}
\begin{lemma}\label{le:duality}
Let $\Gamma$ be a Carleson Jordan curve, let $p:\Gamma\to(1,\infty)$ be a
continuous function satisfying {\rm(\ref{eq:Dini-Lipschitz})}, and let $\varrho$
be a Khvedelidze weight satisfying {\rm(\ref{eq:Khvedelidze})}. Suppose
$a\in L^\infty(\Gamma)$ and $a^{-1}\in L^\infty(\Gamma)$.
\begin{enumerate}
\item[{\rm (a)}]
The operator $aP+Q$ is $n$-normal on $L^{p(\cdot)}(\Gamma,\varrho)$ if and only
if the operator $a^{-1}P+Q$ is $d$-normal on $L^{p^*(\cdot)}(\Gamma,\varrho^{-1})$.
In this case
\begin{equation}\label{eq:duality-1}
n\big(aP+Q;L^{p(\cdot)}(\Gamma,\varrho)\big)
=
d\big(a^{-1}P+Q;L^{p^*(\cdot)}(\Gamma,\varrho^{-1})\big).
\end{equation}

\item[{\rm (b)}]
The operator $aP+Q$ is $d$-normal on $L^{p(\cdot)}(\Gamma,\varrho)$ if and only
if the operator $a^{-1}P+Q$ is $n$-normal on $L^{p^*(\cdot)}(\Gamma,\varrho^{-1})$.
In this case
\[
d\big(aP+Q;L^{p(\cdot)}(\Gamma,\varrho)\big)
=
n\big(a^{-1}P+Q;L^{p^*(\cdot)}(\Gamma,\varrho^{-1})\big).
\]
\end{enumerate}
\end{lemma}
\begin{proof}
By Theorem~\ref{th:linear-functional}, the space $L^{p^*(\cdot)}(\Gamma,\varrho^{-1})$
may be identified with the dual space $[L^{p(\cdot)}(\Gamma,\varrho)]^*$.
Let us prove part (a). The operator $aP+Q$ is $n$-normal on $L^{p(\cdot)}(\Gamma,\varrho)$
if and only if its adjoint $(aP+Q)^*$ is $d$-normal on the dual space
$L^{p^*(\cdot)}(\Gamma,\varrho^{-1})$ and
\begin{equation}\label{eq:duality-2}
n\big(aP+Q;L^{p(\cdot)}(\Gamma,\varrho)\big)
=
d\big((aP+Q)^*;L^{p^*(\cdot)}(\Gamma,\varrho^{-1})\big).
\end{equation}
 From Theorem~\ref{th:linear-functional} it follows that
\begin{equation}\label{eq:duality-4}
(aI)^*=H_\Gamma aH_\Gamma.
\end{equation}
Combining Lemma~\ref{le:S-adjoint} and (\ref{eq:duality-4}), we get
\begin{equation}\label{eq:duality-5}
(aP+Q)^*=H_\Gamma(P+QaI)H_\Gamma.
\end{equation}
On the other hand, taking into account Lemma~\ref{le:basic},
it is easy to check that
\begin{equation}\label{eq:duality-6}
P+QaI=(I+Pa^{-1}Q)(a^{-1}P+Q)(I-Qa^{-1}P)aI,
\end{equation}
where $I+Pa^{-1}Q$, $I-Qa^{-1}P$, and $aI$ are invertible operators
on $L^{p^*(\cdot)}(\Gamma,\varrho^{-1})$. From (\ref{eq:duality-5}) and
(\ref{eq:duality-6}) it follows that $(aP+Q)^*$ and $a^{-1}P+Q$ are
$d$-normal on the space $L^{p^*(\cdot)}(\Gamma,\varrho^{-1})$ only
simultaneously and
\begin{equation}\label{eq:duality-7}
d\big((aP+Q)^*;L^{p^*(\cdot)}(\Gamma,\varrho^{-1})\big)=
d\big(a^{-1}P+Q;L^{p^*(\cdot)}(\Gamma,\varrho^{-1})\big).
\end{equation}
Combining (\ref{eq:duality-2}) and (\ref{eq:duality-7}), we arrive at
(\ref{eq:duality-1}). Part (a) is proved.
The proof of part (b) is analogous.
\end{proof}
Denote by $L_{N\times N}^\infty(\Gamma)$ the algebra of all $N\times N$ matrix
functions with entries in the space $L^\infty(\Gamma)$.
\begin{lemma}\label{le:duality-matrix}
Let $\Gamma$ be a Carleson Jordan curve, let $p:\Gamma\to(1,\infty)$ be a
continuous function satisfying {\rm (\ref{eq:Dini-Lipschitz})}, and let
$\varrho$ be a Khvedelidze weight satisfying {\rm (\ref{eq:Khvedelidze})}.
Suppose $a\in L_{N\times N}^\infty(\Gamma)$ and $a^T$ is the transposed matrix
of $a$. Then the operator $P+aQ$ is $n$-normal (resp. $d$-normal)
on $L_N^{p(\cdot)}(\Gamma,\varrho)$ if and only if the operator $a^TP+Q$ is
$d$-normal (resp. $n$-normal) on $L_N^{p^*(\cdot)}(\Gamma,\varrho^{-1})$.
\end{lemma}
\begin{proof}
In view of Theorem~\ref{th:linear-functional}, the space
$L_N^{p^*(\cdot)}(\Gamma,\varrho^{-1})$ may be identified with the dual space
$[L_N^{p(\cdot)}(\Gamma,\varrho)]^*$, and the general form of a linear functional
on $L_N^{p(\cdot)}(\Gamma,\varrho)$ is given by
\[
G(f)=\sum_{j=1}^N\int_\Gamma f_j(\tau)\overline{g_j(\tau)}\,|d\tau|,
\]
where $f=(f_1,\dots,f_N)\in L_N^{p(\cdot)}(\Gamma,\varrho)$ and
$g=(g_1,\dots,g_N)\in L_N^{p^*(\cdot)}(\Gamma,\varrho^{-1})$, and the norms
in $[L_N^{p(\cdot)}(\Gamma,\varrho)]^*$ and in $L_N^{p^*(\cdot)}(\Gamma,\varrho^{-1})$
are equivalent. It is easy to see that $(aI)^*=H_\Gamma a^TH_\Gamma$, where
$H_\Gamma$ is defined on $L_N^{p^*(\cdot)}(\Gamma,\varrho^{-1})$ elementwise.
From Lemma~\ref{le:S-adjoint} it follows that $P^*=H_\Gamma QH_\Gamma$ and
$Q^*=H_\Gamma PH_\Gamma$ on $L_N^{p^*(\cdot)}(\Gamma,\varrho^{-1})$. Then
\begin{equation}\label{eq:duality-matrix-1}
(P+aQ)^*=H_\Gamma(Pa^TI+Q)H_\Gamma.
\end{equation}
On the other hand, it is easy to see that
\begin{equation}\label{eq:duality-matrix-2}
Pa^TI+Q=(I+Pa^TQ)(a^TP+Q)(I-Qa^TP),
\end{equation}
where the operators $I+Pa^TQ$ and $I-Qa^TP$ are invertible on
$L_N^{p^*(\cdot)}(\Gamma,\varrho^{-1})$. From (\ref{eq:duality-matrix-1})
and (\ref{eq:duality-matrix-2}) it follows that $(P+aQ)^*$ and $a^TP+Q$
are $n$-normal (resp. $d$-normal) on $L_N^{p^*(\cdot)}(\Gamma,\varrho^{-1})$
only simultaneously. This implies the desired statement.
\end{proof}
\section{Closedness of the image of $aP+Q$ in the scalar case}
\label{sec:closed-range}
\subsection{Functions in $L^{p(\cdot)}(\Gamma,\varrho)$
are better than integrable if $S$ is bounded}
\begin{lemma}\label{le:embedding}
Suppose $\Gamma$ is a Carleson Jordan curve and $p:\Gamma\to(1,\infty)$ is a
continuous function satisfying {\rm (\ref{eq:Dini-Lipschitz})}. If $\varrho$
is a Khvedelidze weight satisfying {\rm (\ref{eq:Khvedelidze})}, then there
exists an $\eps>0$ such that $L^{p(\cdot)}(\Gamma,\varrho)$ is continuously
embedded in $L^{1+\eps}(\Gamma)$.
\end{lemma}
\begin{proof}
If (\ref{eq:Khvedelidze}) holds, then there exists a number $\eps>0$ such that
\[
0<(1/p(t_k)+\lambda_k)(1+\eps)<1
\quad\mbox{for all}\quad k\in\{1,\dots,m\}.
\]
Hence, by Theorem~\ref{th:KPS}, the operator $S$ is bounded on
$L^{p(\cdot)/(1+\eps)}(\Gamma,\varrho^{1+\eps})$. In that case the operator
$\varrho^{1+\eps}S\varrho^{-1-\eps}I$ is bounded on $L^{p(\cdot)/(1+\eps)}(\Gamma)$.
Obviously, the operator $V$ defined by $(Vg)(t)=tg(t)$ is bounded on
$L^{p(\cdot)/(1+\eps)}(\Gamma)$, and
\[
((AV-VA)g)(t)=\frac{\varrho^{1+\eps}(t)}{\pi i}
\int_\Gamma\frac{g(\tau)}{\varrho^{1+\eps}(\tau)}\,d\tau.
\]
Since $AV-VA$ is bounded on $L^{p(\cdot)/(1+\eps)}(\Gamma)$, there exists a
constant $C>0$ such that
\[
\left|\int_\Gamma\frac{g(\tau)}{\varrho^{1+\eps}(\tau)}\,d\tau\right|
\|\varrho^{1+\eps}\|_{p(\cdot)/(1+\eps)}
=
\left\|
\varrho^{1+\eps}\int_\Gamma\frac{g(\tau)}{\varrho^{1+\eps}(\tau)}\,d\tau
\right\|_{p(\cdot)/(1+\eps)}
\le
C\|g\|_{p(\cdot)/(1+\eps)}
\]
for all $g\in L^{p(\cdot)/(1+\eps)}(\Gamma)$. Since $\varrho(\tau)>0$
a.e. on $\Gamma$, we have $\|\varrho^{1+\eps}\|_{p(\cdot)/(1+\eps)}>0$.
Hence
\[
\Lambda(g)=\int_\Gamma\frac{g(\tau)}{\varrho^{1+\eps}(\tau)}
e^{i\Theta_\Gamma(\tau)}\,|d\tau|
\]
is a bounded linear functional on $L^{p(\cdot)/(1+\eps)}(\Gamma)$. From
Theorem~\ref{th:linear-functional} it follows that
$\varrho^{-1-\eps}\in L^{[p(\cdot)/(1+\eps)]^*}(\Gamma)$, where
\[
\left(\frac{p(t)}{1+\eps}\right)^*=\frac{p(t)}{p(t)-(1+\eps)}
\]
is the conjugate exponent for $p(\cdot)/(1+\eps)$. By Theorem~\ref{th:Hoelder},
\begin{equation}\label{eq:embedding-1}
\int_\Gamma|f(\tau)|^{1+\eps}|d\tau|
\le
C_{p(\cdot),\eps}
\left\|\,|f|^{1+\eps}\varrho^{1+\eps}\right\|_{p(\cdot)/(1+\eps)}
\|\varrho^{-1-\eps}\|_{[p(\cdot)/(1+\eps)]^*}.
\end{equation}
It is easy to see that
\begin{equation}\label{eq:embedding-2}
\left\|\,|f|^{1+\eps}\varrho^{1+\eps}\right\|_{p(\cdot)/(1+\eps)}
=
\|f\varrho\|_{p(\cdot)}^{1+\eps}
=
\|f\|_{p(\cdot),\varrho}^{1+\eps}.
\end{equation}
From (\ref{eq:embedding-1}) and (\ref{eq:embedding-2}) it follows that
$\|f\|_{1+\eps}\le C_{p(\cdot),\eps,\varrho}\|f\|_{p(\cdot),\varrho}$
for all $f\in L^{p(\cdot)}(\Gamma,\varrho)$, where
$C_{p(\cdot),\eps,\varrho}:=
(C_{p(\cdot),\eps}\|\varrho^{-1-\eps}\|_{[p(\cdot)/(1+\eps)]^*})^{1/(1+\eps)}<\infty$.
\end{proof}
\subsection{Criterion for Fredholmness of $aP+Q$ in the scalar case}
\begin{theorem}[see \cite{Karlovich05}, Theorem~3.3]
\label{th:criterion-Fredholmness-scalar}
Let $\Gamma$ be a Carleson Jordan curve satisfying the logarithmic whirl
condition {\rm(\ref{eq:spiralic})} at each point $t\in\Gamma$, let
$p:\Gamma\to(1,\infty)$ be a continuous function satisfying
{\rm(\ref{eq:Dini-Lipschitz})}, and let $\varrho$ be a Khvedelidze weight
satisfying {\rm(\ref{eq:Khvedelidze})}. Suppose $a\in PC(\Gamma)$. The
operator $aP+Q$ is Fredholm on $L^{p(\cdot)}(\Gamma,\varrho)$ if and only
if $a(t\pm 0)\ne 0$ and
\begin{equation}\label{eq:Fredholm}
-\frac{1}{2\pi}\arg\frac{a(t-0)}{a(t+0)}
+
\frac{\delta(t)}{2\pi}\log\left|\frac{a(t-0)}{a(t+0)}\right|
+
\frac{1}{p(t)}+\lambda(t)\notin\Z
\end{equation}
for all $t\in\Gamma$, where
\[
\lambda(t):=\left\{
\begin{array}{lcl}
\lambda_k, &\mbox{if} & t=t_k, \quad k\in\{1,\dots,m\},\\
0,         &\mbox{if} & t\notin\Gamma\setminus\{t_1,\dots,t_m\}.
\end{array}
\right.
\]
\end{theorem}
The necessity portion of this result was obtained in \cite[Theorem~8.1]{Karlovich03}
for spaces with variable exponents satisfying \eqref{eq:Dini-Lipschitz}
under the assumption that $S$ is bounded on $L^{p(\cdot)}(\Gamma,w)$,
where $\Gamma$ is an arbitrary rectifiable Jordan curve and $w$ is an arbitrary
weight (not necessarily power). The sufficiency portion follows from
\cite[Lemma~7.1]{Karlovich03} and Theorem~\ref{th:KPS} (see \cite{Karlovich05}
for details). The restriction \eqref{eq:spiralic} comes up in the proof
of the sufficiency portion because under this condition one can guarantee
the boundedness of the weighted operator $wSw^{-1}I$, where
$w(\tau)=|(t-\tau)^\gamma|$ and $\gamma\in\C$. If $\Gamma$ does not satisfy
\eqref{eq:spiralic}, then the weight $w$ is not equivalent to a Khvedelidze
weight and Theorem~\ref{th:KPS} is not applicable to the operator $wSw^{-1}I$,
that is, a more general result than Theorem~\ref{th:KPS} is needed to treat
the case of arbitrary Carleson curves. As far as we know, such a result is not
known in the case of variable exponents. For a constant exponent $p$, the result of
Theorem~\ref{th:criterion-Fredholmness-scalar} (for arbitrary Muckenhoupt
weights) is proved in \cite{BK95} (see also \cite[Proposition~7.3]{BK97}
for the case of arbitrary Muckenhoupt weights and arbitrary Carleson curves).
\subsection{Criterion for the closedness of the image of $aP+Q$}
\begin{theorem}\label{th:criterion-closedness}
Let $\Gamma$ be a Carleson Jordan curve satisfying the logarithmic whirl
condition {\rm(\ref{eq:spiralic})} at each point $t\in\Gamma$, let
$p:\Gamma\to(1,\infty)$ be a continuous function satisfying
{\rm(\ref{eq:Dini-Lipschitz})}, and let $\varrho$ be a Khvedelidze weight
satisfying {\rm(\ref{eq:Khvedelidze})}. Suppose $a\in PC(\Gamma)$ has
finitely many jumps and $a(t\pm 0)\ne 0$ for all $t\in\Gamma$. Then
the image of $aP+Q$ is closed in $L^{p(\cdot)}(\Gamma,\varrho)$ if and only
if {\rm(\ref{eq:Fredholm})} holds for all $t\in\Gamma$.
\end{theorem}
\begin{proof}
The idea of the proof is borrowed from \cite[Proposition~7.16]{BK97}.
The sufficiency part follows from Theorem~\ref{th:criterion-Fredholmness-scalar}.
Let us prove the necessity part. Assume that $a(t\pm 0)\ne 0$ for all
$t\in\Gamma$. Since the number of jumps, that is, the points $t\in\Gamma$ at
which $a(t-0)\ne a(t+0)$, is finite, it is clear that
\[
\begin{split}
&
-\frac{1}{2\pi}\arg\frac{a(t-0)}{a(t+0)}
+\frac{\delta(t)}{2\pi}\log\left|\frac{a(t-0)}{a(t+0)}\right|
+\frac{1}{1+\eps}\notin\Z,
\\[3mm]
&
-\frac{1}{2\pi}\arg\frac{a(t+0)}{a(t-0)}
+\frac{\delta(t)}{2\pi}\log\left|\frac{a(t+0)}{a(t-0)}\right|
+\frac{1}{1+\eps}\notin\Z
\end{split}
\]
for all $t\in\Gamma$ and all sufficiently small $\eps>0$. By
Theorem~\ref{th:criterion-Fredholmness-scalar}, the operators $aP+Q$
and $a^{-1}P+Q$ are Fredholm on the Lebesgue space $L^{1+\eps}(\Gamma)$
whenever $\eps>0$ is sufficiently small. From Lemma~\ref{le:embedding}
it follows that we can pick $\eps_0>0$ such that
\[
L^{p(\cdot)}(\Gamma,\varrho)\subset L^{1+\eps_0}(\Gamma),
\quad
L^{p^*(\cdot)}(\Gamma,\varrho^{-1})\subset L^{1+\eps_0}(\Gamma)
\]
and $aP+Q$, $a^{-1}P+Q$ are Fredholm on $L^{1+\eps_0}(\Gamma)$. Then
\begin{equation}\label{eq:criterion-closedness-1}
n\big(aP+Q;L^{p(\cdot)}(\Gamma,\varrho)\big)
\le
n\big(aP+Q;L^{1+\eps_0}(\Gamma)\big)<\infty,
\end{equation}
and taking into account Lemma~\ref{le:duality}(b),
\begin{equation}\label{eq:criterion-closedness-2}
\begin{split}
d\big(aP+Q;L^{p(\cdot)}(\Gamma,\varrho)\big)
&=
n\big(a^{-1}P+Q;L^{p^*(\cdot)}(\Gamma,\varrho^{-1})\big)
\\
&\le
n\big(a^{-1}P+Q;L^{1+\eps_0}(\Gamma)\big)<\infty.
\end{split}
\end{equation}
If (\ref{eq:Fredholm}) does not hold, then $aP+Q$ is not Fredholm on
$L^{p(\cdot)}(\Gamma,\varrho)$ in view of Theorem~\ref{th:criterion-Fredholmness-scalar}.
From this fact and (\ref{eq:criterion-closedness-1})--(\ref{eq:criterion-closedness-2})
we conclude that the image of $aP+Q$ is not closed in $L^{p(\cdot)}(\Gamma,\varrho)$,
which contradicts the hypothesis.
\end{proof}
\section{Necessary condition for semi-Fredholmness of $aP+bQ$.\\ The matrix case}
\label{sec:necessity}
\subsection{Two lemmas on approximation of measurable matrix functions}
Let the algebra $L_{N\times N}^\infty(\Gamma)$ be equipped with the norm
\[
\|a\|_{L_{N\times N}^\infty(\Gamma)}:=N\max_{1\le i,j\le N}\|a_{ij}\|_{L^\infty(\Gamma)}.
\]
\begin{lemma}[see \cite{LS87}, Lemma~3.4]
\label{le:approx1}
Let $\Gamma$ be a rectifiable Jordan curve. Suppose $a$ is a measurable
$N\times N$ matrix function on $\Gamma$ such that
$a^{-1}\notin L_{N\times N}^\infty(\Gamma)$. Then for every $\eps>0$ there
exists a matrix function $a_\eps\in L_{N\times N}^\infty(\Gamma)$ such that
$\|a_\eps\|_{L_{N\times N}^\infty(\Gamma)}<\eps$ and the matrix function
$a-a_\eps$ degenerates on a subset $\gamma\subset\Gamma$ of positive measure.
\end{lemma}
\begin{lemma}[see \cite{LS87}, Lemma~3.6]
\label{le:approx2}
Let $\Gamma$ be a rectifiable Jordan curve. If $a$ belongs to
$L_{N\times N}^\infty(\Gamma)$, then for every $\eps>0$ there exists an
$a_\eps\in L_{N\times N}^\infty(\Gamma)$ such that
$\|a-a_\eps\|_{L_{N\times N}^\infty(\Gamma)}<\eps$ and
$a_\eps^{-1}\in L_{N\times N}^\infty(\Gamma)$.
\end{lemma}
\subsection{Necessary condition for $d$-normality of $aP+Q$ and $P+aQ$}
\begin{lemma}\label{le:d-normality}
Suppose $\Gamma$ is a Carleson Jordan curve, $p:\Gamma\to(1,\infty)$ is a
continuous function satisfying {\rm(\ref{eq:Dini-Lipschitz})}, and $\varrho$
is a Khvedelidze weight satisfying {\rm(\ref{eq:Khvedelidze})}. If
$a\in L_{N\times N}^\infty(\Gamma)$ and at least one of the operators
$aP+Q$ or $P+aQ$ is $d$-normal on $L_N^{p(\cdot)}(\Gamma,\varrho)$, then
$a^{-1}\in L_{N\times N}^\infty(\Gamma)$.
\end{lemma}
\begin{proof}
This lemma is proved by analogy with \cite[Theorem~3.13]{LS87}.
For definiteness, let us consider the operator $P+aQ$. Assume that
$a^{-1}\notin L_{N\times N}^\infty(\Gamma)$. By Lemma~\ref{le:approx1},
for every $\eps>0$ there exists an $a_\eps\in L_{N\times N}^\infty(\Gamma)$
such that $\|a-a_\eps\|_{L_{N\times N}^\infty(\Gamma)}<\eps$ and $a_\eps$
degenerates on a subset $\gamma\subset\Gamma$ of positive measure. We have
\[
\|(P+aQ)-(P+a_\eps Q)\|_{\cB(L_N^{p(\cdot)}(\Gamma,\varrho))}
\le
\|a-a_\eps\|_{L_{N\times N}^\infty(\Gamma)}\|Q\|_{\cB(L_N^{p(\cdot)}(\Gamma,\varrho))}
=O(\eps)
\]
as $\eps\to 0$. Hence there is an $\eps>0$ such that $P+a_\eps Q$ is $d$-normal
together with $P+aQ$ due to Theorem~\ref{th:stability}. Since the image
of the operator $P+a_\eps Q$ is a subspace of finite codimension in
$L_N^{p(\cdot)}(\Gamma,\varrho)$, it has a nontrivial intersection with any
infinite-dimensional linear manifold contained in $L_N^{p(\cdot)}(\Gamma,\varrho)$.
In particular, the image of $P+a_\eps Q$ has a nontrivial intersection with linear
manifolds $M_j$, $j\in\{1,\dots,N\}$, of those vector-functions, the $j$-th
component of which is a polynomial of $1/z$ vanishing at infinity and all the
remaining components are identically zero. That is, there exist
\[
\psi_j^+\in PL_N^{p(\cdot)}(\Gamma,\varrho),
\quad
\psi_j^-\in QL_N^{p(\cdot)}(\Gamma,\rho),
\quad
h_j\in M_j,
\quad
h_j\not\equiv 0
\]
such that $\psi_j^++a_\eps\psi_j^-=h_j$ for all $j\in\{1,\dots,N\}$. Consider
the $N\times N$ matrix functions
\[
\Psi_+:=[\psi_1^+,\psi_2^+,\dots,\psi_N^+],
\quad
\Psi_-:=[\psi_1^-,\psi_2^-,\dots,\psi_N^-],
\quad
H:=[h_1,h_2,\dots,h_N],
\]
where $\psi_j^+$, $\psi_j^-$, and $h_j$ are taken as columns. Then
$H-\Psi_+=a_\eps\Psi_-$. Therefore,
\[
\det(H-\Psi_+)=\det a_\eps\det\Psi_-
\quad\mbox{a.e. on}\quad\Gamma.
\]
The left-hand side of this equality is a meromorphic function having a pole
at zero of at least $N$-th order. Thus, it is not identically zero in $D_+$.

On the other hand, each entry of $H-\Psi_+$ belongs to
\[
PL^{p(\cdot)}(\Gamma,\varrho)+\cR\subset E^1(D_+)+\cR
\]
(see Lemma~\ref{le:Hardy-Smirnov}). Hence, by Lemma~\ref{le:Hoelder-Smirnov},
$\det(H-\Psi_+)\in E^{1/N}(D_+)+\cR$ and $\det(H-\Psi_+)$
degenerates on $\gamma$ because $a_\eps$ degenerates on $\gamma$.
In view of Theorem~\ref{th:Lusin-Privalov}, $\det (H-\Psi_+)$
vanishes identically in $D_+$. This is a contradiction. Thus,
$a^{-1}$ belongs to $L_{N\times N}^\infty(\Gamma)$.
\end{proof}
\subsection{Necessary condition for semi-Fredholmness of $aP+bQ$}
\begin{theorem}\label{th:semi-Fredholmness-necessity}
Let $\Gamma$ be a Carleson Jordan curve, let $p:\Gamma\to(1,\infty)$ be a
continuous function satisfying {\rm (\ref{eq:Dini-Lipschitz})}, and let $\varrho$
be a Khvedelidze weight satisfying {\rm (\ref{eq:Khvedelidze})}.
If the coefficients $a$ and $b$ belong to $L_{N\times N}^\infty(\Gamma)$ and
the operator $aP+bQ$ is semi-Fredholm on
$L_N^{p(\cdot)}(\Gamma,\varrho)$, then $a^{-1},b^{-1}\in L_{N\times N}^\infty(\Gamma)$.
\end{theorem}
\begin{proof}
The proof is analogous to the proof of \cite[Theorem~3.18]{LS87}.
Suppose $aP+bQ$ is $d$-normal on $L_N^{p(\cdot)}(\Gamma,\varrho)$.
By Lemma~\ref{le:approx2}, for every $\eps>0$ there exist
$a_\eps\in L_{N\times N}^\infty(\Gamma)$ such that
$a_\eps^{-1}\in L_{N\times N}^\infty(\Gamma)$ and
$\|a-a_\eps\|_{L_{N\times N}^\infty(\Gamma)}<\eps$. Since
\[
\|(aP+bQ)-(a_\eps P+bQ)\|_{\cB(L_N^{p(\cdot)}(\Gamma,\varrho))}
\le
\|a-a_\eps\|_{L_{N\times N}^\infty(\Gamma)}
\|P\|_{\cB(L_N^{p(\cdot)}(\Gamma,\varrho))}
=O(\eps)
\]
as $\eps\to 0$, from Theorem~\ref{th:stability} it follows
that $\eps>0$ can be chosen so small that $a_\eps P+bQ$ is $d$-normal
on $L_N^{p(\cdot)}(\Gamma,\varrho)$, too. Since
$a_\eps^{-1}\in L_{N\times N}^\infty(\Gamma)$, the operator $a_\eps I$ is
invertible on $L_N^{p(\cdot)}(\Gamma,\varrho)$. From
Theorem~\ref{th:Atkinson-Yood} it follows that the operator
$P+a_\eps^{-1}bQ=a_\eps^{-1}(a_\eps P+ bQ)$ is $d$-normal. By
Lemma~\ref{le:d-normality}, $b^{-1}a_\eps$ belongs to $L_{N\times N}^\infty(\Gamma)$.
Hence $b^{-1}=b^{-1}a_\eps a_\eps^{-1}\in L_{N\times N}^\infty(\Gamma)$.

Furthermore, $b^{-1}aP+Q=b^{-1}(aP+bQ)$ and the operator $b^{-1}aP+Q$
is $d$-normal on $L_N^{p(\cdot)}(\Gamma,\varrho)$. By Lemma~\ref{le:d-normality},
$a^{-1}b\in L_{N\times N}^\infty(\Gamma)$. Then $a^{-1}=a^{-1}bb^{-1}$
belongs to $L_{N\times N}^\infty(\Gamma)$. That is, we have shown
that if $aP+bQ$ is $d$-normal on $L_N^{p(\cdot)}(\Gamma,\varrho)$, then
$a^{-1},b^{-1}\in L_{N\times N}^\infty(\Gamma)$.

If $aP+bQ$ is $n$-normal on $L_N^{p(\cdot)}(\Gamma,\varrho)$, then arguing
as above, we conclude that the operator $P+a_\eps^{-1}bQ$ is $n$-normal on
$L_N^{p(\cdot)}(\Gamma,\varrho)$. By Lemma~\ref{le:duality-matrix},
the operator $(a_\eps^{-1}b)^TP+Q$ is $d$-normal on
$L_N^{p^*(\cdot)}(\Gamma,\varrho^{-1})$. From Lemma~\ref{le:d-normality}
it follows that $[(a_\eps^{-1}b)^T]^{-1}\in L_{N\times N}^\infty(\Gamma)$.
Therefore, $b^{-1}=(a_\eps^{-1})^{-1}a_\eps^{-1}\in L_{N\times N}^\infty(\Gamma)$.
Furthermore, $b^{-1}aP+Q=b^{-1}(aP+bQ)$ and the operator $b^{-1}aP+Q=b^{-1}(aP+bQ)$
is $n$-normal on $L_N^{p(\cdot)}(\Gamma,\varrho)$. From Lemma~\ref{le:duality-matrix}
we get that the operator $P+(b^{-1}a)^TQ$ is $d$-normal on
$L_N^{p^*(\cdot)}(\Gamma,\varrho^{-1})$. Applying Lemma~\ref{le:d-normality} to
the operator $P+(b^{-1}a)^TQ$ acting on $L_N^{p^*(\cdot)}(\Gamma,\varrho^{-1})$,
we obtain $a^{-1}b\in L_{N\times N}^\infty(\Gamma)$. Thus
$a^{-1}=a^{-1}bb^{-1}\in L_{N\times N}^\infty(\Gamma)$.
\end{proof}
\section{Semi-Fredholmness and Fredholmness of $aP+bQ$ are equivalent}
\label{sec:SIO-matrix}
\subsection{Decomposition of piecewise continuous matrix functions}
Denote by $PC^0(\Gamma)$ the set of all piecewise continuous functions $a$
which have only a finite number of jumps and satisfy $a(t-0)=a(t)$ for all
$t\in\Gamma$. Let $C_{N\times N}(\Gamma)$ and $PC_{N\times N}^0(\Gamma)$ denote
the sets of $N\times N$ matrix functions with continuous entries  and with
entries in $PC^0(\Gamma)$, respectively. A matrix function
$a\in PC_{N\times N}(\Gamma)$ is said to be nonsingular if $\det a(t\pm 0)\ne 0$
for all $t\in\Gamma$.
\begin{lemma}[see \cite{CG81}, Chap.~VII, Lemma~2.2]
\label{le:decomposition}
Suppose $\Gamma$ is a rectifiable Jordan curve. If a matrix function
$f\in PC_{N\times N}^0(\Gamma)$ is nonsingular, then there exist an upper-triangular
nonsingular matrix function $g\in PC_{N\times N}^0(\Gamma)$ and nonsingular
matrix functions $c_1,c_2\in C_{N\times N}(\Gamma)$ such that $f=c_1gc_2$.
\end{lemma}
\subsection{Compactness of commutators}
\begin{lemma}\label{le:compactness}
Let $\Gamma$ be a Carleson Jordan curve, let $p:\Gamma\to(1,\infty)$ be a
continuous function satisfying {\rm(\ref{eq:Dini-Lipschitz})}, and let $\varrho$
be a Khvedelidze weight satisfying {\rm(\ref{eq:Khvedelidze})}.
If $c$ belongs to $C_{N\times N}(\Gamma)$, then the commutators $cP-PcI$
and $cQ-QcI$ are compact on $L_N^{p(\cdot)}(\Gamma,\varrho)$.
\end{lemma}
This statement follows from Theorem~\ref{th:KPS} and \cite[Lemma~6.5]{Karlovich03}.
\subsection{Equivalence of semi-Fredholmness and Fredholmness of $aP+bQ$}
\begin{theorem}\label{th:criterion-Fredholmness-matrix}
Let $\Gamma$ be a Carleson Jordan curve satisfying the logarithmic whirl
condition {\rm(\ref{eq:spiralic})} at each point $t\in\Gamma$, let
$p:\Gamma\to(1,\infty)$ be a continuous function satisfying
{\rm(\ref{eq:Dini-Lipschitz})}, and let $\varrho$ be a Khvedelidze weight
satisfying {\rm(\ref{eq:Khvedelidze})}. If $a,b\in PC_{N\times N}^0(\Gamma)$, then
$aP+bQ$ is semi-Fredholm on $L_N^{p(\cdot)}(\Gamma,\varrho)$
if and only if it is Fredholm on $L_N^{p(\cdot)}(\Gamma,\varrho)$.
\end{theorem}
\begin{proof}
The idea of the proof is borrowed from \cite[Theorem~3.1]{Spitkovsky92}.
Only the necessity portion of the theorem is nontrivial. If $aP+bQ$ is
semi-Fredholm, then $a$ and $b$ are nonsingular by
Theorem~\ref{th:semi-Fredholmness-necessity}. Hence $b^{-1}a$ is nonsingular.
In view of Lemma~\ref{le:decomposition}, there exist an upper-triangular
nonsingular matrix function $g\in PC_{N\times N}^0(\Gamma)$ and continuous
nonsingular matrix functions $c_1$, $c_2$ such that $b^{-1}a=c_1gc_2$.
It is easy to see that
\begin{equation}\label{eq:criterion-Fredholmness-matrix-1}
aP+bQ=bc_1\big[(gP+Q)(Pc_2I+Qc_1^{-1}I)+g(c_2P-Pc_2I)+(c_1^{-1}Q-Qc_1^{-1}I)\big].
\end{equation}
From Lemma~\ref{le:compactness} it follows that the operators $c_2P-Pc_2I$
and $c_1^{-1}Q-Qc_1^{-1}I$ are compact on $L_N^{p(\cdot)}(\Gamma,\varrho)$ and
\[
(Pc_2I+Qc_1^{-1}I)(c_2^{-1}P+c_1Q)=I+K_1,
\quad
(c_2^{-1}P+c_1Q)(Pc_2I+Qc_1^{-1}I)=I+K_2,
\]
where $K_1$ and $K_2$ are compact operators on $L_N^{p(\cdot)}(\Gamma,\varrho)$.
In view of these equalities, by Theorem~\ref{th:regularization},
the operator $Pc_2I+Qc_1^{-1}I$ is Fredholm on $L_N^{p(\cdot)}(\Gamma,\varrho)$.
Obviously, the operator $bc_1I$ is invertible because $bc_1$ is nonsingular.
From (\ref{eq:criterion-Fredholmness-matrix-1}) and Theorem~\ref{th:Atkinson-Yood}
it follows that $aP+bQ$ is $n$-normal, $d$-normal, Fredholm if and only if
$gP+Q$ has the corresponding property.

Let $g_j$, $j\in\{1,\dots,N\}$, be the elements of the main diagonal of the
upper-triangular matrix function $g$. Since $g$ is nonsingular, all $g_j$
are nonsingular, too. Assume for definiteness that $gP+Q$ is $n$-normal
on $L_N^{p(\cdot)}(\Gamma,\varrho)$. By Theorem~\ref{th:Spitkovsky}(a),
the operator $g_1P+Q$ is $n$-normal on $L^{p(\cdot)}(\Gamma,\varrho)$.
Hence the image of $g_1P+Q$ is closed. From Theorem~\ref{th:criterion-closedness}
it follows that (\ref{eq:Fredholm}) is fulfilled with $g_1$ in place of $a$.
Therefore, the operator $g_1P+Q$ is Fredholm on $L^{p(\cdot)}(\Gamma,\varrho)$
due to Theorem~\ref{th:criterion-Fredholmness-scalar}. Applying
Theorem~\ref{th:Spitkovsky}(b), we deduce that the operator $g^{(1)}P+Q$
is $n$-normal on $L_{N-1}^{p(\cdot)}(\Gamma,\varrho)$, where $g^{(1)}$ is
the $(N-1)\times (N-1)$ upper-triangular nonsingular matrix function obtained from
$g$ by deleting the first column and the first row.
Arguing as before with $g^{(1)}$ in place of $g$, we conclude that $g_2P+Q$ is
Fredholm on $L^{p(\cdot)}(\Gamma,\varrho)$ and $g^{(2)}P+Q$ is $n$-normal
on $L_{N-2}^{p(\cdot)}(\Gamma,\varrho)$, where $g^{(2)}$ is the
$(N-2)\times(N-2)$ upper-triangular nonsingular matrix function obtained
from $g^{(1)}$ by deleting the first column and the first row. Repeating
this procedure $N$ times, we can show that all operators $g_jP+Q$,
$j\in\{1,\dots,N\}$, are Fredholm on $L^{p(\cdot)}(\Gamma,\varrho)$.

If the operator $gP+Q$ is $d$-normal, then we can prove in a similar fashion
that all operators $g_jP+Q$, $j\in\{1,\dots,N\}$, are Fredholm on
$L^{p(\cdot)}(\Gamma,\varrho)$. In this case we start with $g_N$ and delete
the last column and the last row of the matrix $g^{(j-1)}$ on the $j$-th step
(we assume that $g^{(0)}=g$).

Since all operators $g_jP+Q$ are Fredholm on $L^{p(\cdot)}(\Gamma,\varrho)$,
from Theorem~\ref{th:Spitkovsky}(b) we obtain that the operator $gP+Q$ is Fredholm
on $L_N^{p(\cdot)}(\Gamma,\varrho)$. Hence $aP+bQ$ is Fredholm on
$L_N^{p(\cdot)}(\Gamma,\varrho)$, too.
\end{proof}
\section{Semi-Fredholmness and Fredholmness are equivalent for arbitrary
operators in $\alg(S,PC,L_N^{p(\cdot)}(\Gamma,\varrho))$}
\label{sec:algebra}
\subsection{Linear dilation}
The following statement shows that the semi-Fredholmness of an operator
in a dense subalgebra of $\alg(S,PC,L_N^{p(\cdot)}(\Gamma,\varrho))$ is
equivalent to the semi-Fredholmness of a simpler operator $aP+bQ$
with coefficients of $a,b$ of larger size.
\begin{lemma}\label{le:dilation}
Suppose $\Gamma$ is a Carleson Jordan curve, $p:\Gamma\to(1,\infty)$ is a
continuous function satisfying {\rm (\ref{eq:Dini-Lipschitz})}, and $\varrho$
is a Khvedelidze weight satisfying {\rm (\ref{eq:Khvedelidze})}.
Let
\[
A=\sum_{i=1}^k A_{i1}A_{i2}\dots A_{ir},
\]
where $A_{ij}=a_{ij}P+b_{ij}Q$ and all $a_{ij},b_{ij}$ belong to
$PC_{N\times N}^0(\Gamma)$. Then there exist functions
$a,b\in PC_{D\times D}^0(\Gamma)$, where $D:=N(k(r+1)+1)$,
such that $A$ is $n$-normal ($d$-normal, Fredholm) on
$L_N^{p(\cdot)}(\Gamma,\varrho)$ if and only
if $aP+bQ$ is $n$-normal (resp. $d$-normal, Fredholm) on
$L_{D}^{p(\cdot)}(\Gamma,\varrho)$.
\end{lemma}
\begin{proof}
The idea of the proof is borrowed from \cite{GK71} (see also
\cite[Theorem~12.15]{BBKS99}). Denote by $O_s$ and $I_s$ the $s\times s$
zero and identity matrix, respectively. For $\ell=1,\dots,r$, let $B_\ell$
be the $kN\times kN$ matrix
\[
B_\ell=\mathrm{diag}(A_{1\ell},A_{2\ell},\dots,A_{k\ell}),
\]
then define the $kN(r+1)\times kN(r+1)$ matrix $Z$ by
\[
Z=\left[\begin{array}{ccccc}
I_{kN} & B_1    & O_{kN} & \dots & O_{kN} \\
O_{kN} & I_{kN} & B_2    & \dots & O_{kN} \\
\vdots & \vdots & \vdots & \ddots & \vdots\\
O_{kN} & O_{kN} & O_{kN} & \dots & B_r \\
O_{kN} & O_{kN} & O_{kN} & \dots & I_{kN}
\end{array}\right].
\]
Put
\[
X:=\mathrm{column}(\underbrace{O_N,\dots,O_N}_{kr},\underbrace{-I_N,\dots,-I_N}_k),
\quad
Y:=(\underbrace{I_N,\dots,I_N}_k,\underbrace{O_N,\dots,O_N}_{kr}).
\]
Define also $M_0=(\underbrace{I_N,\dots,I_N}_k)$ and for $\ell\in\{1,\dots,r\}$, let
\[
M_\ell:=(A_{11}A_{12}\dots A_{1\ell}\,,\,A_{21}A_{22}\dots A_{2\ell}\,,\,
\dots\,,\,A_{k1}A_{k2}\dots A_{k\ell}).
\]
Finally, put
\[
W:=(M_0,M_1,\dots,M_r).
\]
It can be verified straightforwardly that
\begin{equation}\label{eq:dilation}
\left[\begin{array}{cc}
I_{kN(r+1)} & O \\ W & I_N
\end{array}\right]
\left[\begin{array}{cc}
I_{kN(r+1)} & O \\ O & A
\end{array}\right]
\left[\begin{array}{cc}
Z & X \\ O & I_N
\end{array}\right]
=
\left[\begin{array}{cc}
Z & X \\ Y & O_N
\end{array}\right].
\end{equation}
It is clear that the outer terms on the left-hand side of (\ref{eq:dilation})
are invertible. Hence the middle factor of (\ref{eq:dilation}) and the
right-hand side of (\ref{eq:dilation}) are $n$-normal ($d$-normal, Fredholm)
only simultaneously in view of Theorem~\ref{th:Atkinson-Yood}. By
Theorem~\ref{th:Spitkovsky}(b), the operator $A$ is $n$-normal ($d$-normal,
Fredholm) if and only if the middle factor of (\ref{eq:dilation}) has the
corresponding property. Finally, note that the left-hand side of (\ref{eq:dilation})
has the form $aP+bQ$, where $a,b\in PC_{D\times D}^0(\Gamma)$.
\end{proof}
\subsection{Proof of Theorem~\ref{th:main}}
Obviously, for every $f\in PC(\Gamma)$ there exists a sequence $f_n\in PC^0(\Gamma)$
such that $\|f-f_n\|_{L^\infty(\Gamma)}\to 0$ as $n\to\infty$. Therefore,
for each operator $\alpha P+\beta Q$, where $\alpha=(\alpha_{rs})_{r,s=1}^N$,
$\beta=(\beta_{rs})_{rs=1}^N$ and $\alpha_{rs},\beta_{rs}\in PC(\Gamma)$
for all $r,s\in\{1,\dots,N\}$, there exist sequences
$\alpha^{(n)}=(\alpha_{rs}^{(n)})_{r,s=1}^N$, $\beta^{(n)}=(\beta_{rs}^{(n)})_{r,s=1}^N$
with $\alpha_{rs}^{(n)},\beta_{rs}^{(n)}\in PC^0(\Gamma)$ for all
$r,s\in\{1,\dots,N\}$ such that
\[
\begin{split}
&
\|(\alpha P+\beta Q)-(\alpha^{(n)}P+\beta^{(n)}Q)\|_{\cB(L_N^{p(\cdot)}(\Gamma,\varrho))}
\\
&\le
N\max_{1\le r,s\le N}\|\alpha_{rs}-\alpha_{rs}^{(n)}\|_{L^\infty(\Gamma)}
\|P\|_{\cB(L_N^{p(\cdot)}(\Gamma,\varrho))}
\\
&\quad
+N\max_{1\le r,s\le N}\|\beta_{rs}-\beta_{rs}^{(n)}\|_{L^\infty(\Gamma)}
\|Q\|_{\cB(L_N^{p(\cdot)}(\Gamma,\varrho))}=o(1)
\end{split}
\]
as $n\to\infty$.

Let $A\in\alg(S,PC;L_N^{p(\cdot)}(\Gamma,\varrho))$. Then there exists a
sequence of operators $A^{(n)}$ of the form
$\sum_{i=1}^kA_{i1}^{(n)}A_{i2}^{(n)}\dots A_{ir}^{(n)}$, where
$A_{ij}^{(n)}=a_{ij}^{(n)}P+b_{ij}^{(n)}Q$ and
$a_{ij}^{(n)},b_{ij}^{(n)}$ belong to $PC_{N\times N}(\Gamma)$,
such that $\|A-A^{(n)}\|_{\cB(L_N^{p(\cdot)}(\Gamma,\varrho))}\to 0$ as
$n\to\infty$. In view of what has been said above, without loss of generality,
we can assume that all matrix functions $a_{ij}^{(n)},b_{ij}^{(n)}$ belong
to $PC_{N\times N}^0(\Gamma)$.

If $A$ is semi-Fredholm, then for all sufficiently large $n$, the operators
$A^{(n)}$ are semi-Fredholm by Theorem~\ref{th:stability}. From Lemma~\ref{le:dilation}
it follows that for every semi-Fredholm operator
$\sum_{i=1}^kA_{i1}^{(n)}A_{i2}^{(n)}\dots A_{ir}^{(n)}$ there exist
$a^{(n)},b^{(n)}\in PC_{D\times D}^0(\Gamma)$, where $D:=N(k(r+1)+1)$,
such that $a^{(n)}P+b^{(n)}Q$ is semi-Fredholm on $L_{D}^{p(\cdot)}(\Gamma,\varrho)$.
By Theorem~\ref{th:criterion-Fredholmness-matrix}, $a^{(n)}P+b^{(n)}Q$
is Fredholm on $L_{D}^{p(\cdot)}(\Gamma,\varrho)$. Applying
Lemma~\ref{le:dilation} again, we conclude that
$\sum_{i=1}^kA_{i1}^{(n)}A_{i2}^{(n)}\dots A_{ir}^{(n)}$ is Fredholm on
$L_N^{p(\cdot)}(\Gamma,\varrho)$. Thus, for all sufficiently large $n$,
the operators $A^{(n)}$ are Fredholm. Lemma~\ref{le:stability} yields that
$A$ is Fredholm.
\qed

\end{document}